\documentclass[a4paper]{amsart}
\usepackage[utf8]{inputenc}

\usepackage{amssymb}
\usepackage{amsmath}
\usepackage{amsthm}
\usepackage{array}
\usepackage{mathtools}
\usepackage{multirow}
\usepackage{tikz}
\usepackage{tikz-cd}
\usepackage{thmtools}
\usepackage{nicefrac}

\usepackage{lipsum}

\usepackage[colorlinks=true, citecolor=blue]{hyperref}
\usepackage{url}

\usepackage{todonotes}

\theoremstyle{plain}
\newtheorem{theorem}{Theorem}[section]
\newtheorem*{theorem*}{Theorem}

\newtheorem{lemma}[theorem]{Lemma}

\theoremstyle{definition}

\theoremstyle{remark}
\newtheorem{remark}[theorem]{Remark}

\newcommand{\ZZ}{\mathbb{Z}}
\newcommand{\Z}{\mathbb{Z}}
\newcommand{\HH}{\mathbb{H}}

\renewcommand{\ge}{\geqslant}
\renewcommand{\le}{\leqslant}

\DeclareMathOperator{\Ker}{Ker}
\DeclareMathOperator{\Isom}{Isom}

\DeclareMathOperator{\tors}{tors}
\DeclareMathOperator{\vol}{vol}

\newcommand{\Ofive}{\frac{\sqrt{5} + 1}{2}}
\DeclareMathOperator{\Stab}{Stab}

\title[Hyperbolic manifolds with exponential homology torsion growth]{Family of hyperbolic manifolds with exponential homology torsion growth}
\author{Stepan Alexandrov}
\address{\parbox{\linewidth}{Max Planck Institute for Mathematics, Bonn, Germany \\ Mathematical Institute, University of Bonn, Germany}}
\email{cyanprism@gmail.com}
\urladdr{cyanprism.github.io}

\begin{document}

\begin{abstract}
	In this note, we construct a family of hyperbolic manifolds with exponentially growing torsion in their homology groups. This demonstrates that the recent bound on homological torsion, established by Bader, Gelander, and Sauer, is asymptotically sharp and cannot be improved.
\end{abstract}

\maketitle

\section{Introduction}

The study of torsion in the homology of arithmetic hyperbolic manifolds has gained significant attention in recent years. In particular, it is known that in many settings, the torsion part of homology can grow exponentially in towers of finite covers. The results of Bergeron--Venkatesh (\cite{BV13}), Raimbault (\cite{Rai12}), and others provide asymptotic lower bounds for torsion in the homology of arithmetic manifolds under strong assumptions.

In this note, we focus on a recent result of Bader--Gelander--Sauer.

\begin{theorem}[\cite{BGS20}]
	For every \(n \ne 3\), there exists \(C = C_n > 0\) such that for every complete \(n\)-dimensional Riemannian manifold \(M\) of normalised bounded negative curvature and for every degree \(i,\) \[\log |\tors H_i(M; \mathbb{Z})| \leq C \cdot \vol(M).\]
\end{theorem}

We will prove that the bound cannot be improved in the following sense.

\begin{theorem} \label{theorem: main}
	For every \(n \geq 3\), there exists a sequence of compact hyperbolic manifolds \(M^n_p = \HH^n / \Gamma^n_p\) such that \([\Gamma^n_1 : \Gamma^n_p] = p\) and \[\log_2 |\tors H_i(M^n_p; \ZZ)| \geq p, \quad \text{for all } i = 1, \dots, n - 2.\]
\end{theorem}

More precisely, for each \(n \ge 3\), we construct a sequence of compact orientable arithmetic hyperbolic \(n\)-manifolds of simplest type \(M^n_p\) such that \[H_i(M^n_p; \Z) \supseteq (\Z / 2\Z)^p\] for all \(i = 1, \dots, n - 2\), where each \(M^n_p\) is a \(p\)-fold cover of a fixed manifold \(M^n_1\). 

\subsection*{Structure of the paper}

In Section~\ref{section: preliminaries}, we collect the necessary background on arithmetic hyperbolic manifolds and right-angled Coxeter groups used throughout the paper. The proof of Theorem~\ref{theorem: main} is given in Section~\ref{section: main theorem proof} and consists of two parts: the case of dimension three, and the case of higher dimensions. Each part is based on a key lemma, which are proved in Sections~\ref{section: base case} and~\ref{section: inductive step}, respectively.


\subsection*{Acknowledgements}

The author is grateful to Ursula Hamenst\"adt for her supervision and invaluable support, and to Nikolay Bogachev and Sami Douba for helpful discussions.

\section{Preliminaries}
\label{section: preliminaries}

\subsection{Arithmetic hyperbolic manifolds of simplest type}

The present survey follows \cite[Section~12.8]{Rat19}.

Let \(f\) be a quadratic form in \(n\) variables with real symmetric coefficient matrix \(A = (a_{ij})\). Then we have \(f(x) = x^t A x\). Let \(R\) be a subring of \(\mathbb{R}\). We say that \(f\) is \emph{over} \(R\) if \(a_{ij} \in R\) for all \(i, j\). The \emph{orthogonal group} of \(f\) over \(R\) is defined to be \begin{multline*}\operatorname{O}(f, R) = \{T \in \operatorname{GL}(n, R) \mid f(Tx) = f(x) \text{ for all } x \in \mathbb{R}^n\} \\ = \{T \in \operatorname{GL}(n, \mathbb{R}) \mid T^t A T = A\}.\end{multline*}

Let \(\ell_n\) be the \emph{Lorentzian quadratic form} in \(n + 1\) variables given by \(\ell_n(x) = -x_0^2 + x_1^2 + \dots + x_n^2\) and \(\langle x, y \rangle_n = \frac{1}{4} (\ell_n(x + y) - \ell_n(x - y))\) be the associated bilinear form. Then \(\operatorname{O}(\ell_n, \mathbb{R}) = \operatorname{O}(n, 1)\). The hyperboloid model of \emph{hyperbolic \(n\)-space} is \[\mathbb{H}^n = \{x \in \mathbb{R}^{n + 1} \mid \ell_n(x) = -1 \text{ and } x_0 > 0\}.\] The restriction of \(\langle \cdot, \cdot \rangle_n\) to the tangent space of \(\mathbb{H}^n\) at any point is positive define and, therefore, \(\mathbb{H}^n\) is a Riemannian manifold. Let \(\operatorname{O}^+(n, 1)\) be the subgroup of \(\operatorname{O}(n, 1)\) consisting of all \(T \in \operatorname{O}(n, 1)\) that leave \(\mathbb{H}^n\) invariant. Then \(\operatorname{O}^+(n, 1)\) has index \(2\) in \(\operatorname{O}(n, 1)\). Restriction induces an isomorphism from \(\operatorname{O}^+(n, 1)\) to \(\operatorname{Isom}(\mathbb{H}^n)\). We shall identify \(\operatorname{O}^+(n, 1)\) with the group of isometries of \(\mathbb{H}^n\).

Let \(\Gamma < \operatorname{O}^+(n, 1)\) be a discrete subgroup; equivalently, the orbit of each point \(x \in \mathbb{H}^n\) is discrete. A discrete subgroup \(\Gamma\) is called a \emph{(hyperbolic) lattice} if there exists a Borel subset \(D \subset \mathbb{H}^n\) such that \(\operatorname{vol}(B) < +\infty\) and \(\bigcup_{\gamma \in \Gamma} \gamma D = \mathbb{H}^n\). If \(D\) is compact, then \(\Gamma\) is said to be \emph{cocompact}.

A \emph{hyperbolic \(n\)-manifold} is a complete, connected Riemannian \(n\) manifold of constant sectional curvature \(-1\). Every complete, connected, simply-connected manifold of constant negative curvature \(-1\) is isometric to \(\mathbb{H}^n\). Thus, every hyperbolic manifold \(M\) is isometric to \(\mathbb{H}^n / \Gamma\), where \(\Gamma < \operatorname{O}^+(n, 1)\) is a torsion-free discrete subgroup, which is isomorphic to \(\pi_1(M)\). The manifold has finite volume if and only if \(\Gamma\) is a lattice. Moreover, \(M\) is compact if and only if \(\Gamma\) is cocompact. Selberg’s lemma, which asserts that every finitely generated matrix group contains a torsion-free subgroup of finite index, can be used to construct hyperbolic manifolds from lattices.

A \emph{number field} \(k\) is a subfield of \(\mathbb{C}\) that is an extension of \(\mathbb{Q}\) of finite degree. A number field \(k\) is said to be \emph{totally real} if all the field embeddings of \(k\) into \(\mathbb{C}\) take values in \(\mathbb{R}\).

Suppose that the quadratic form \(f\) has signature \((n, 1)\). This means that there exists \(M \in \operatorname{GL}(n, \mathbb{R})\) such that \(f(Mx) = \ell_n(x)\) for all \(x \in \mathbb{R}^{n + 1}\). Let \(\operatorname{O}^+(f, R)\) be the subgroup of \(\operatorname{O}(f, R)\) consisting of all \(T \in \operatorname{O}(f, R)\) that leave both components of \(\{x \in \mathbb{R}^{n + 1} \mid f(x) < 0\}\) invariant. Then \(\operatorname{O}^+(f, R)\) has index \(2\) in \(\operatorname{O}(f, R)\).

The group \(\operatorname{O}^+(f, \mathbb{R})\) is a topological group with respect to the Euclidean metric topology on \(\operatorname{GL}(n + 1, \mathbb{R})\). We have that \[M_* \colon \operatorname{O}^+(n, 1) \to \operatorname{O}^+(f, \mathbb{R}), \quad T \mapsto M T M^{-1},\] is an isomorphism of topological groups.

Let \(k\) be a totally real number field, and let \(f\) be a quadratic form over \(k\) in \(n + 1\) variables, with \(n > 0\) and symmetric coefficient matrix \(A = (a_{ij})\). The quadratic form \(f\) is said to be \emph{admissible} if \(f\) has signature \((n, 1),\) and for each nonidentity field embedding, \(\sigma \colon k \to \mathbb{R},\) the quadratic form \(f^\sigma\) over \(\sigma(k)\), with coefficient matrix \(A^\sigma = (\sigma(a_{ij}))\), is positive definite.

Subgroups \(H_1\) and \(H_2\) of a group \(G\) are said to be \emph{commensurable} if \(H_1 \cap H_2\) has finite index in both \(H_1\) and \(H_2\). Let \(\mathcal{O}_k = \mathbb{A} \cap k,\) where \(\mathbb{A}\) denotes the ring of all algebraic integers, be the \emph{ring of integers} of \(k\). A subgroup \(\Gamma\) of \(\operatorname{O}^+(n, 1)\) is called an \emph{arithmetic group of isometries of \(\mathbb{H}^n\) of simplest type defined over} a totally real number field \(K\) if there exists an admissible quadratic form \(f\) over \(k\) in \(n + 1\) variables and a matrix \(M \in \operatorname{GL}(n + 1, \mathbb{R})\) such that \[f(Mx) = \ell_n(x) \quad \text{for all } x \in \mathbb{R}^{n + 1}\] and the subgroups \(M \Gamma M^{-1}\) and \(\operatorname{O}^+(f, \mathcal{O}_k)\) of \(\operatorname{O}^+(f, \mathbb{R})\) are commensurable.


According to the celebrated result by Borel and Harish-Chandra, every arithmetic group is a lattice. Moreover, if \(k \ne \mathbb{Q},\) then \(\Gamma\) is cocompact.

A \emph{\(\Gamma\)-hyperplane} is a hyperplane \(H \subset \HH^n\) such that the quotient \(H / \Stab_\Gamma(H)\) is compact; equivalently, \(\Stab_\Gamma(H)\) is a lattice in \(\Isom(H)\). If \(\Gamma\) is torsion-free, then \(S = H / \Stab_\Gamma(H)\) is an immersed totally geodesic hypersurface in \(M = \HH^n / \Gamma\). Since the universal cover of any hyperbolic manifold is contractible, continuous retractions \(M \to S\) (up to homotopy) are in one-to-one correspondence with homomorphisms \(\Gamma \to \Stab_\Gamma(H)\) that restrict to the identity on \(\Stab_\Gamma(H)\). This motivates the following definition.

Let \(\Gamma\) be a lattice and \(H\) a \(\Gamma\)-hyperplane. Then a \emph{retraction} is a homomorphism \(r \colon \Gamma \to \Stab_\Gamma(H)\) that restricts to the identity on \(\Stab_\Gamma(H)\). Note that if such a homomorphism exists, then it induces an embedding \(H_k(\Stab_\Gamma(H)) \hookrightarrow H_k(\Gamma)\).

\subsection{Hyperbolic right-angled Coxeter groups} It is well known that every totally geodesic subspace of \(\mathbb{H}^n\) can be realised as the non-trivial intersection of \(\mathbb{H}^n\) with a linear subspace \(V \subset\mathbb{R}^{n + 1}\). For a vector \(e \in \mathbb{R}^{n + 1}\) satisfying \(\ell_n(e) = 1\), we define the hyperplane \[H_e^0 = \{x \in \mathbb{H}^n \mid \langle e, x \rangle = 0\}\] and the corresponding closed half-space \[H_e^- = \{x \in \mathbb{H}^n \mid \langle e, x \rangle \le 0\}.\] The dihedral angle \(\phi\) between \(H_{e_1}^0\) and \(H_{e_2}^0\) is determined by \[\langle e_1, e_2 \rangle = -\cos\phi.\]


A \emph{(hyperbolic) polytope} \(P \subseteq \mathbb{H}^n\) is a finite intersection of half-spaces. We further assume that \(P\) is compact (and hence \(\operatorname{vol}(P) < +\infty\)) and has non-empty interior (\(\dim P = n\)). A polytope is said to be \emph{right-angled} if all of its dihedral angles are equal to \(\frac{\pi}{2}\).

Every right-angled polytope \(P\) determines a \emph{Coxeter group} \(\Gamma(P)\) generated by reflections in the facets of \(P\). This group is discrete, and \(P\) serves as a fundamental domain for \(\Gamma(P)\); that is, \(\bigcup_{\gamma \in \Gamma(P)} \gamma P = \mathbb{H}^n\) and \[\forall\gamma \in \Gamma(P) \setminus \{1\} \quad \operatorname{int}P \cap \operatorname{int} \gamma P = \varnothing.\] Since \(P\) is compact, \(\Gamma(P)\) is a cocompact lattice.

Although right-angled polytopes possess many desirable properties (to be discussed below), the difficulty is that no compact right-angled polytopes exist in \(\HH^n\) for \(n > 4\) (\cite{PV05}).

Let \(P\) be a compact right-angled hyperbolic polytope. 
The group \(\Gamma(P)\) has the following presentation: 
\[
	\Gamma(P) = \left\langle \gamma_f \text{ for every facet } f \;\middle|\;
	\begin{aligned}
		& \gamma_f^2 = 1 && \text{for every facet } f, \\
		& \gamma_{f_1}\gamma_{f_2} = \gamma_{f_2}\gamma_{f_1} && \text{if \(f_1\) and \(f_2\) are adjacent}
	\end{aligned}
	\right\rangle.
\]
Let \(H\) be a hyperplane containing a facet \(f\). The subgroup \(\Stab_{\Gamma(P)}(H)\) is generated by reflection in \(H\) and in all hyperplanes containing facets adjacent to \(f\). A retraction \(\Gamma \to \Stab_{\Gamma(P)}(H)\) can be constracted by sending to the identity all generators of \(\Gamma(P)\) that do not belong to \(\Stab_{\Gamma(P)}(H)\).

Finally, there is a simple construction of a finite-index torsion-free subgroup of \(\Gamma(P)\) in dimension \(3\) (see, for example, \cite[Section~3]{Ves17}). Let \(S\) be a finite subgroup of \(\Gamma(P)\), where \(P \subset \HH^3\). Then there exists a vertex \(v\) of the polytope \(P\) and an element \(\gamma \in \Gamma(P)\) such that \(\gamma S \gamma^{-1} \subseteq \Stab_{\Gamma(P)}(v)\). Thus, if \(Q\) is a finite group, \(\phi \colon \Gamma(P) \to Q\) is a homomorphism, and \(\phi|_{\Stab_{\Gamma(P)}(v)} = \operatorname{id}\) for every vertex \(v\), then \(\Ker \phi\) is a finite-index torsion-free subgroup.

The four colour theorem implies that the faces of \(P \subset \HH^3\) can be coloured with four colours so that no two adjacent faces share the same colour. Let \(\alpha\), \(\beta\), and \(\gamma\) denote a basis of the vector space \((\ZZ / 2\ZZ)^3\), and set \(\delta = \alpha + \beta + \gamma\). Note that any three among \(\alpha\), \(\beta\), \(\gamma\), and \(\delta\) are linearly independent. Now let us consider a homomorphism \(\phi \colon \Gamma(P) \to (\ZZ / 2\ZZ)^3\) that sends each generator \(\gamma_f\) to one of the vectors \(\alpha\), \(\beta\), \(\gamma\), or \(\delta\) according to the colour of the facet \(f\). For this map, we have \(\phi|_{\Stab_{\Gamma(P)}(v)} = \operatorname{id}\) for every vertex \(v\), and therefore \(\Ker \phi\) is torsion-free. Moreover, \(\Ker \phi\) contains no orientation-reversing elements, and hence \(\HH^3 / \Ker\phi\) is orientable.

\section{Proof of Theorem~\ref{theorem: main}}
\label{section: main theorem proof}

Let us recall that our goal is to construct, for each \(n \ge 3\), a sequence of compact orientable arithmetic hyperbolic \(n\)-manifolds of simplest type \(M^n_p\) such that \[H_i(M^n_p; \Z) \supseteq (\Z / 2\Z)^p\] for all \(i = 1, \dots, n - 2\), where each \(M^n_p\) is a \(k\)-fold cover of a fixed \(M^n_1\). 

The construction proceeds by induction on the dimension \(n\). The base case is \(n = 3\). The inductive step consists of constructing examples of dimension \(n\) from examples of dimension \(n - 1\).

\subsection{The base case}


The following lemma provides the base case. Since the proof is not short, it will be proved in Section~\ref{section: base case}.


\begin{lemma} \label{lemma: base case}
	There exists a family of compact orientable arithmetic hyperbolic \(3\)-manifolds of simplest type \(M^3_p = \HH^3 / \Gamma^3_p\) such that \([\Gamma^3_1 : \Gamma^3_p] = p\) and \[H_1(M^3_p; \Z) \supseteq (\ZZ / 2\ZZ)^p.\]
\end{lemma}

The idea of the proof is to construct manifolds \(M^3_k\) with the following properties:
\begin{itemize}
	\item the manifold \(M^3_p\) contains \(p\) non-orientable subsurfaces;
	\item there exist retractions of \(M^3_p\) onto each of these subsurfaces;
	\item these retractions are ``independent'' in the sense that the fundamental group of each subsurface maps trivially under the retraction onto any other subsurface.
\end{itemize}
The first property implies that the first homology group of each subsurface contains a copy of \((\Z/2\Z)\). The second ensures that the first homology group of each subsurface injects into \(H_1(M^3_p; \Z)\). Finally, the third property implies that the images of these subgroups intersect trivially. Thus, the torsion parts of the first homology groups of subsurfaces form a subgroup isomorphic to \((\Z / 2\Z)^p\) inside \(H_1(M^3_p; \Z)\).

\subsection{The inductive step}

For \(n \ge 4\) there exists a family of compact orientable arithmetic hyperbolic \((n - 1)\)-manifolds of simplest type \(M^{n - 1}_p = \HH^{n - 1} / \Gamma^{n - 1}_p\) such that \([\Gamma^{n - 1}_1 : \Gamma^{n - 1}_p] = p\) and \[H_i(M^{n - 1}_p; \Z) \supseteq (\ZZ / 2\ZZ)^p \quad \text{for all } i = 1, \dots, n - 3.\] As \(M^{n - 1}_1\) is an arithmetic manifold of simplest type, there exist an admissible quadratic form \(q_{n - 1}\) over a totally real number field \(k\) such that \(\Gamma^{n - 1}_1\) is a finite-index subgroup of \(\operatorname{O}^+(q_{n - 1}, \mathcal{O}_k)\).

\begin{remark}
	We do not use it explicitly, but our construction uses the quadratic form \(q_{n - 1} = -\Ofive x_0^2 + x_1^2 + \dots + x_{n - 1}^2\), defined over \(k = \mathbb{Q}[\sqrt{5}]\), with \(\mathcal{O}_k = \mathbb{Z}[\Ofive]\).
\end{remark}


The following lemma is essential for our proof and will be proved in Section~\ref{section: inductive step}.

\begin{lemma} \label{lemma: inductive step}
	There exists a compact arithmetic manifold \(M^n_1 = \HH^n / \Gamma^n_1\) with the following properties:
	\begin{itemize}
		\item \(M^{n - 1}_1\) is embedded as a totally geodesic submanifold in \(M^n_1\);
		\item there exists a retraction \(r_n \colon \Gamma^n_1 \to \Gamma^{n - 1}_1\).
	\end{itemize}
\end{lemma}

Given this, define \(\Gamma^n_p = r_n^{-1}(\Gamma^{n - 1}_p)\) and \(M^n_p = \HH^n / \Gamma^n_p\). As a retraction induces an embedding of the homology groups, we have \[(\Z / 2\Z)^p \subseteq H_i(M^{n - 1}_p; \Z) = H_i(\Gamma^{n - 1}_p; \Z) \subseteq H_i(\Gamma^n_p; \Z) = H_i(M^n_p; \Z)\] for all \(i = 1, \dots, n - 3\). Moreover, by the Poincar\'e duality, we also have \[(\Z / 2\Z)^p \subseteq \tors H_1(M^n_p; \Z) = \tors H_{n - 2}(M^n_p; \Z),\] which completes the proof of the theorem.

\section{Proof of Lemma~\ref{lemma: base case}}
\label{section: base case}

\begin{figure}
	\includegraphics[scale=1]{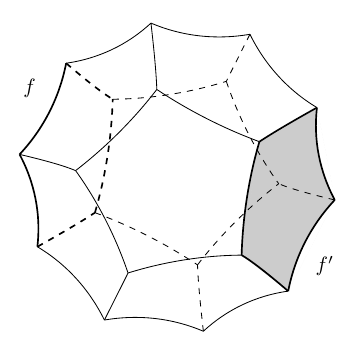}
	\caption{The regular right-angled dodecahedron.}
	\label{figure: dodecahedron}
\end{figure}

Let \(P \subset \HH^3\) be the regular right-angled hyperbolic dodecahedron. The group \(\Gamma(P)\) is a subgroup of an arithmetic Coxeter group \(\Gamma^3\) (see, for example, \cite[Section~2]{Ves17}). Since every arithmetic Coxeter group is an arithmetic group of simplest type, the same holds for \(\Gamma(P)\). In fact, it is shown in \cite{Bug84} that \[\Gamma^3 = \operatorname{O}^+\left(-\frac{1 + \sqrt{5}}{2} x_0^2 + x_1^2 + x_2^2 + x_3^2,\ \Z\left[\frac{1 + \sqrt{5}}{2}\right]\right).\]

Let \(f\) and \(f'\) denote a pair of opposite faces of \(P\) (shown in bold in Figure~\ref{figure: dodecahedron}), and \(H\) and \(H'\) denote the hyperplanes containing \(f\) and \(f'\) respectively. We remind that there are the retractions \[r \colon \Gamma(P) \to \Stab_{\Gamma(P)}(H) \quad \text{and} \quad r' \colon \Gamma(P) \to \Stab_{\Gamma(P)}(H')\] that map to the identity all generators of \(\Gamma(P)\) that do not belong to \(\Stab_{\Gamma(P)}(H)\) and \(\Stab_{\Gamma(P)}(H')\) respectively. Moreover,  as \(f\) and \(f'\) have no common adjacent faces, \[r\left(\Stab_{\Gamma(P)}(H')\right) = 1 \quad \text{and} \quad r'\left(\Stab_{\Gamma(P)}(H)\right) = 1.\]

%
\begin{figure}[]
	\includegraphics{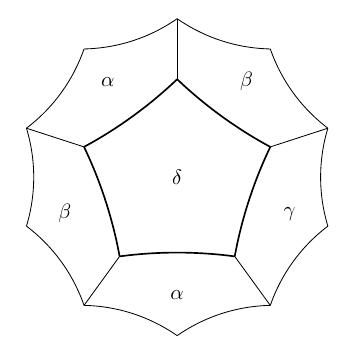}
	\includegraphics{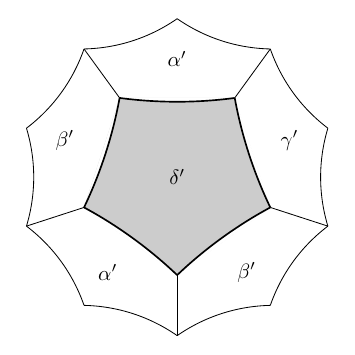}
	\caption{The face colouring of the dodecahedron.}
	\label{figure: face colouring}
\end{figure}

Let \(\alpha\), \(\beta\), \(\gamma\), and \(\delta\) denote the standard basis of \((\Z / 2\Z)^4 \oplus 0\), and let \(\alpha'\), \(\beta'\), and \(\gamma'\) denote the standard basis of \(0 \oplus (\Z / 2\Z)^3\). Define \(\delta' = \alpha' + \beta' + \gamma'\). Let \[\phi \colon \Gamma(P_1) \to (\Z / 2\Z)^4 \oplus (\Z / 2\Z)^3\] that maps a generator of \(\Gamma(P)\) to the colour of the corresponding face in Figure~\ref{figure: face colouring}. Note that any three elements of the set \(\{\alpha, \beta, \gamma, \delta, \alpha', \beta', \gamma', \delta'\}\) are linearly independent, and no two adjacent faces have the same colour. Therefore, the kernel \(\Ker\phi\) is torsion-free and \(\HH^3 / \Ker\phi\) is a compact manifold. Moreover, as \(\sigma \circ \phi\), where \(\sigma\) is the sum of all coordinates, maps all generators of \(\Gamma(P)\), which are reflections, to \(1\), the kernel of \(\phi\) is torsion-free and the manifold \(\HH^3 / \Ker\phi\) is orientable.

We denote \(\Gamma = \Ker \phi\). The restrictions \(r|_\Gamma\) and \(r'|_\Gamma\) define retractions \[\Gamma \to \Stab_\Gamma(H) \quad \text{and} \quad \Gamma \to \Stab_\Gamma(H')\] respectively. Indeed, let \(\rho\) and \(\rho'\) denote the projections from \((\Z / 2\Z)^4 \oplus (\Z / 2\Z)^3\) onto the first and second summands respectively. Then \[\rho \circ \phi = \phi \circ r \quad \text{and} \quad \rho' \circ \phi = \phi \circ r'.\] Thus, \(\phi(\gamma) = 0\) implies \(\phi(r(\gamma)) = 0\) and \(\phi(r'(\gamma)) = 0\). It follows that  \(r(\Gamma) \subset \Gamma\) and \(r'(\Gamma) \subset \Gamma\), and hence \(r|_\Gamma\) and \(r'|_\Gamma\) are retractions. 

The last thing to note is that \(S = H / \Stab_\Gamma(H)\) is an orientable surface and \(S' = H' / \Stab_\Gamma(H')\) is non-orientable.
Indeed, \(\sigma \circ \phi\), where \(\sigma\) is the sum of the first three coordinates, maps the reflection in \(f\) to \(0\) and the other generators of \(\Stab_\Gamma(H)\), which inverse the orientation of \(H\), to \(1\). Thus, all elements of \(\Gamma(P)\) that inverse the orientation of \(H\) have a non-trivial image under \(\phi\), and \(S\) is orientable. The homomorphism \(\phi\) maps the product \(\pi\) of the reflections in \(f'\) and the three top faces in the right part of Figure~\ref{figure: face colouring} to \[\delta' + \beta' + \alpha' + \gamma' = 0.\] However, \(\pi\) does not preserve the orientation of \(H'\): the reflection in \(f'\) preserves the orientation of \(H'\), while the other three do not. Thus, \(S'\) is non-orientable.

\begin{figure*}[ht]
\centering
\begin{tabular}{cc}
    \includegraphics{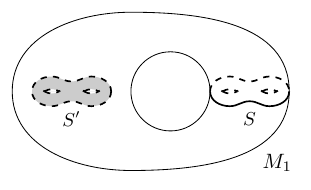} &
    \includegraphics{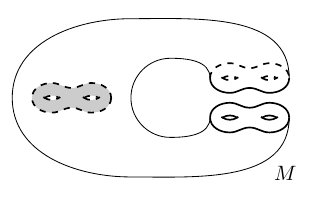} \\
    \includegraphics{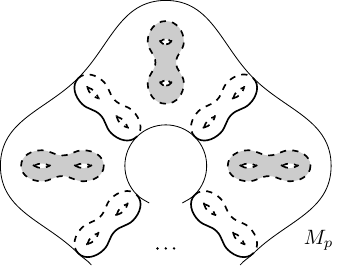} &
    \includegraphics{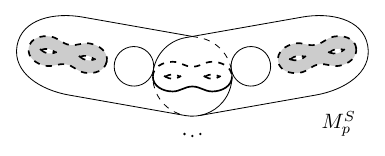}
\end{tabular}
\caption{The manifolds \(M_1\), \(M\), \(M_p\), and \(M^S_p\).}
\label{figure: manifolds}
\end{figure*}

Denote \(M_1 = \HH^3 / \Gamma\). Since both \(M_1\) and \(S\) are orientable, the cohomology class \([S] \in H^1(M; \Z)\) is non-trivial. Let \(M_p\) denote the \(p\)-fold cyclic cover of \(M_1\) associated with the cohomology class \([S]\) (see Figure~\ref{figure: manifolds}). Such covers admit a geometric interpretation: let \(M\) denote the manifold obtained from \(M_1\) by cutting along \(S\). Its boundary \(\partial M\) consists of two copies of \(S\), and \(M_p\) is obtained by gluing \(p\) copies of \(M\) cyclically along their boundary components via the identity map.

Let \(M^{S}_p\) denote the space obtained by gluing \(p\) copies of \(M_1\) along the common embedded surface \(S\) (see Figure~\ref{figure: manifolds}). There is a natural projection \(M_p \twoheadrightarrow M^S_p\) obtained by mapping each copy of \(M \subset M_p\) to the corresponding copy of \(M_1\).

We now show that for every copy of \(S' \subset M_p\) there exists a retraction \(M_p \to S'\) which restricts to the identity on the fundamental group of that copy of \(S'\) and kills the fundamental groups of all other copies. Indeed, fix such a copy of \(S'\). Consider the composition 
\[\begin{tikzcd}
	M_p \arrow[r, two heads] & M^S_p \arrow[r, two heads] & M_1 \arrow[r, two heads, "r'"] & S',
\end{tikzcd}\]
where \begin{itemize}
	\item \(M_p \twoheadrightarrow M_p^{S}\) maps each lifted copy of \(M\) to the corresponding copy of \(M_1\);
	\item \(M_p^{S} \twoheadrightarrow M_1\) collapses all copies of \(M_1\) except the one containing the chosen \(S'\), using the retraction \(r \colon M_1 \to S\) ;
	\item \(r'\colon M_1\twoheadrightarrow S'\) is the retraction constructed earlier.
\end{itemize}
This composite map is a retraction onto the chosen copy of \(S'\) and is trivial on every other copy of \(S'\). Since such a retraction exists for each copy of \(S'\), \[\bigoplus_{i=1}^p H_1(S'; \Z) \subseteq H_1(M_p; \Z).\] Because each \(H_1(S'; \Z)\) contains a non-trivial element of order \(2\), we conclude that \[(\Z/2\Z)^p \subseteq H_1(M_p; \Z).\]

\section{Proof of Lemma~\ref{lemma: inductive step}}
\label{section: inductive step}

Let us recall that \(M^{n - 1}_1 = \HH^{n - 1} / \Gamma^{n - 1}_1\) is a compact arithmetic hyperbolic \((n - 1)\)-manifold of simplest type. Thus, by definition, \(\Gamma^{n - 1}_1\) is a finite-index subgroup of \(\Gamma^{n - 1} = \operatorname{O}^+(q_{n - 1}, \mathcal{O}_k)\) for an admissible quadratic form \(q_{n - 1}\) over a totally real number field \(k\). Let \(q_n = q_{n - 1} + x_n^2\) and \(\Gamma^n = \operatorname{O}^+(q_n, \mathcal{O}_k)\). Note that \(\Gamma^{n - 1} = \Stab_{\Gamma^n}(H)\), where \(H = \{x \in \HH^n \mid x_n = 0\}\). Our goal is to find a finite-index torsion-free subgroup \(\Gamma^n_1\) of \(\Gamma^n\) that contains \(\Gamma^{n - 1}_1\) and admits a retraction \(\Gamma^n_1 \to \Gamma^{n - 1}_1\). To construct the retraction, it is natural to apply the following theorem for \(\Gamma = \Gamma^n\).


\begin{theorem}[\cite{BHW11}] \label{theorem: BHW}
	Let \(\Gamma^n\) be a cocompact arithmetic lattice and \(H \subset \mathbb{H}^n\) be a \(\Gamma^n\)-hyperplane. Then, there exists a finite-index subgroup \(\Gamma'' \subset \Gamma^n\) such that there is a retraction \(r \colon \Gamma'' \to \operatorname{Stab}_{\Gamma''}(H)\).
\end{theorem}

There are two main issues that prevent a direct application of the theorem:
\begin{enumerate}
	\item The group \(\Gamma^n\) is not torsion-free. For instance, it contains the reflection in the hyperplane \(H\).
	\item After passing to a finite-index subgroup \(\Gamma''\) of \(\Gamma^n\), the stabiliser \(\Stab_{\Gamma''}(H)\) does not necessarily contain \(\Gamma^{n - 1}\).
\end{enumerate}
The first issue can, in fact, be reduced to the second. By Selberg’s lemma, there exists a torsion-free finite-index subgroup of \(\Gamma^n\). Thus, we may apply the theorem to such a subgroup in place of \(\Gamma^n\). It then remains to ensure that the finite-index subgroup \(\Gamma''\) provided by the theorem can be chosen in such a way that \(\Stab_{\Gamma’'}(H)\) contains \(\Gamma^{n - 1}_1\). 

Assuming this is possible, we define \(\Gamma^n_1 = r^{-1}(\Gamma^{n - 1}_1)\), where \(r \colon \Gamma'' \to \Stab_{\Gamma''}(H)\) is the retraction provided by the theorem. Then \(\Gamma^n_1\) is a finite-index subgroup of \(\Gamma''\), and hence of \(\Gamma^n\), and satisfies the required properties.

In what follows, we modify the proof of the theorem to ensure that the desired properties are satisfied. We begin by proving the following lemma.



\begin{lemma} \label{lemma: symmetry of retraction}
	In the notation of Theorem~5.1, for any \(\sigma \in \Stab_\Gamma(H)\) there exists a finite-index subgroup \(\Delta'' \subset \Gamma''\) such that \(\sigma \Delta'' \sigma^{-1} = \Delta''\) and
	\[
		r(\sigma \delta \sigma^{-1}) = \sigma r(\delta) \sigma^{-1} \quad \text{for all } \delta \in \Delta''.
	\]
\end{lemma}





\begin{proof}
	Let us at first recap the proof of Theorem~\ref{theorem: BHW}. At first, the authors construct a locally finite \(\Gamma\)-invariant family of \(\Gamma\)-hyperplanes whose dual graph is quasi-isometric to \(\HH^n\). By ``filling in'' the skeletons of the cubes in the dual graph they obtain a \(\operatorname{CAT}(0)\) cube complex \(\mathcal{C}\). Its hyperplanes are in one-to-one correspondence with the constructed locally finite family of hyperplanes in \(\HH^n\).
	
	Next, they take a finite-index subgroup \(\Gamma' \subseteq \Gamma\) and consider the abstract right-angled Coxeter group \(C(\Gamma')\) generated by the \(\Gamma'\)-equivalence classes of these hyperplanes. Two generators commute if and only if the corresponding classes of hyperplanes intersect. Let \(\operatorname{DM}(\Gamma')\) be the Davis--Moussong realisation of the Coxeter group. If \(\Gamma'\) is sufficiently deep then there exists a \(\Gamma'\)-equivariant isometric embedding \(\mathcal{C} \to \operatorname{DM}(\Gamma)\), which provides an embedding \(\Gamma \to C(\Gamma')\): roughly speaking, \(\gamma\) is mapped to the product of the generators that are correspond to the hyperplanes that intersects a representing loop of \(\gamma\).
	
	At last, any abstract right-angled Coxeter group retracts onto the stabiliser of any of the hyperplanes of its Davis--Moussong complex. Using Scott's method, they show that there is a finite-index subgroup \(\Gamma'' \le \Gamma'\) that admits a retraction \(\Gamma'' \to \Stab_{\Gamma''}(H)\).
	
	Let us start with this group \(\Gamma'\). Let \(\Delta' = (\Gamma' \cap \sigma \Gamma' \sigma^{-1}) \subseteq \Gamma'\) be a finite-index subgroup. There exists an automorphism \[\sigma_* \colon \Delta' \to \Delta', \quad \delta \mapsto \sigma \delta \sigma^{-1}.\] This automorphism is induced by the isometry \[\sigma \colon \HH^n / \Delta' \to \HH^n / \Delta', \quad \Delta' x \mapsto \Delta' \sigma x = \sigma \Delta' x.\] Thus, \(\sigma\) defines an isometry of \(\mathcal{C} / \Delta'\) and permutes the generators of \(C(\Delta')\). Let \(\Gamma'' \subseteq \Delta'\) denote the finite-index subgroup with the retraction \(r \colon \Gamma'' \to \Stab_{\Gamma''}(H)\). Let \(\Gamma''' = r^{-1}(\Gamma'' \cap \sigma \Gamma'' \sigma^{-1})\) and \(\Delta'' = \Gamma''' \cap \sigma \Gamma''' \sigma^{-1}\).
	
	Our aim is to prove that the left square of the following diagram commutes.
	\begin{displaymath}
		\begin{tikzcd}[row sep=scriptsize, column sep=scriptsize]
			& \Delta'' \arrow[rr, hook, "i"] \arrow[dd, "r" {yshift=-10pt, swap}] & & C(\Delta') \arrow[dd, "r" {swap}] \\
			\Delta'' \arrow[ur, "\sigma_*"] \arrow[rr, crossing over, hook, "i" {xshift=10pt}] \arrow[dd, "r" {swap}] & & C(\Delta') \arrow[ur, "\sigma_*"] \\
			& \Stab_{\Delta''}(H) \arrow[rr, hook, "i" {xshift=10pt}] & & \Stab_{C(\Delta')}(H) \\
			\Stab_{\Delta''}(H) \arrow[ur, "\sigma_*"] \arrow[rr, hook, "i"] & & \Stab_{C(\Delta')}(H) \arrow[from=uu, crossing over, "r" {yshift=-10pt, swap}] \arrow[ur, "\sigma_*"] \\
		\end{tikzcd}
	\end{displaymath}
	Note that the front and back squares commute, which follows from the definition of \(r \colon \Delta'' \to \Stab_{\Delta''}(H)\). The top, bottom, and right squares commute as \(\sigma\) permutes the generators of \(C(\Delta')\). Finally, the left square commutes as the map
	\begin{displaymath}\begin{tikzcd}
		\Delta'' \arrow[r, "\sigma"] & \Delta'' \arrow[r, "i"] & \Stab_{\Delta''}(H) \arrow[r, hook, "i"] & \Stab_{C(\Delta')}(H)
	\end{tikzcd}\end{displaymath}
	is equal to	
	\begin{displaymath}\begin{tikzcd}
		\Delta'' \arrow[r, "r"] & \Stab_{\Delta''}(H) \arrow[r, "\sigma"] & \Stab_{\Delta''}(H) \arrow[r, hook, "i"] & \Stab_{C(\Delta')}(H).
	\end{tikzcd}\end{displaymath}
\end{proof}

\begin{theorem} \label{theorem: retraction}
	Let \(\Gamma\) be a cocompact arithmetic hyperbolic lattice of simplest type, let \(\Gamma' \subseteq \Gamma\) be a finite-index subgroup, let \(H\) be a \(\Gamma\)-hyperplane, and let \(\Sigma \subseteq \Stab_\Gamma(H)\) be a finite-index subgroup. Then there exists a finite-index subgroup \(\Delta' \subseteq \Gamma\) such that \(\Sigma \subseteq \Delta'\), \(\Delta' \subseteq \langle \Gamma', \Sigma \rangle\), and \(\Stab_{\Delta'}(H) = \Sigma\), and which admits a retraction \(r' \colon \Delta' \to \Stab_{\Delta'}(H)\). Moreover, if both \(\Gamma'\) and \(\Sigma\) are torsion-free, then so is \(\Delta'\).
\end{theorem}

\begin{proof}
	Let us note that the group \(\Stab_\Gamma(H)\), and hence \(\Sigma\), is finitely generated. Let \(\sigma_1, \dots, \sigma_k\) denote the generators of \(\Sigma\). According to Theorem~\ref{theorem: BHW} there is a finite-index subgroup \(\Gamma'' \subset \Gamma'\) that retracts to \(\Stab_{\Gamma''}(H)\). Let \(\Delta_i\), \(\Delta_0\), and \(\Delta_{-i}\) denote the subgroups of \(\Gamma''\) that comes from Lemma~\ref{lemma: symmetry of retraction} applied to \(\sigma_i\), \(\operatorname{id}\), and \(\sigma_i^{-1}\) respectively. Let \(\Delta\) denote the intersection of these subgroups. Then
	\begin{itemize}
		\item \(\Delta\) is a finite-index subgroup of \(\Gamma\);
		\item \(\Delta\) is normalised by \(\Sigma\);
		\item there exists a retraction \(r \colon \Delta \to \Stab_{\Delta}(H)\) such that \[r(\sigma \delta \sigma^{-1}) = \sigma r(\delta) \sigma^{-1} \quad \text{for all } \sigma \in \Sigma \text{ and } \delta \in \Delta.\]
	\end{itemize}
	Since \(\Sigma \cap \Stab_\Delta(H)\) is a finite-index subgroup of \(\Stab_\Delta(H)\), the subgroup \(r^{-1}(\Sigma)\) has finite index in \(\Delta\). We now replace \(\Delta\) by \(r^{-1}(\Sigma)\). The desired group \(\Delta'\) is equal to \(\Delta \Sigma\) and the desired retraction is \[r' \colon \Delta' \to \Stab_{\Delta'}(H), \quad \delta \sigma \mapsto r(\delta) \sigma.\] The retraction is well-defined as 
	\begin{multline*}
		r'(\delta_1 \sigma_1 \, \delta_2 \sigma_2) = 
		r'\!\left( \delta_1 (\sigma_1 \delta_2 \sigma_1^{-1}) \; \sigma_1 \sigma_2 \right) = 
		r\!\left( \delta_1 (\sigma_1 \delta_2 \sigma_1^{-1})\right) \sigma_1 \sigma_2 \\ = 
		r(\delta_1) \, r(\sigma_1 \delta_2 \sigma_1^{-1}) \, \sigma_1 \sigma_2 = 
		r(\delta_1) \sigma_1 \, r(\delta_2) \sigma_2 = 
		r'(\delta_1 \sigma_1) \, r'(\delta_2 \sigma_2).
	\end{multline*}
	And \(\Stab_{\Delta'}(H) = \Sigma\), as if \(\delta \sigma \in \Stab_{\Delta'}(H)\), then \(\delta \in \Stab_\Delta(H) \subseteq \Sigma\) and \(\delta \sigma \in \Sigma\).
	
	Let us now assume that both \(\Gamma'\) and \(\Sigma\) are torsion-free. Note that the following sequence is exact 
	\begin{displaymath}\begin{tikzcd}
		1 \arrow[r] & \Ker r' \arrow[r] & \Delta \Sigma \arrow[r, "r'"] & \Sigma \arrow[r] & 1.
	\end{tikzcd}\end{displaymath}
	If \(r'(\delta \sigma) = r(\delta) \sigma = 1\), then \(\sigma^{-1} = r(\delta) \in \Delta\), which implies \(\sigma \in \Delta\) and \(\delta \sigma \in \Delta\). Thus, \(\Ker r' \subseteq \Delta\). Assume that \(\Delta'\) is not torsion-free. This means that \((\delta \sigma)^m = 1\) for some \(\delta \in \Delta'\), \(\sigma \in \Sigma\), and \(m > 0\). Therefore, \(r'\!\left( (\delta \sigma)^m \right) = r'(\delta \sigma)^m = 1\), but \(r'(\delta \sigma) \in \Sigma\). As \(\Sigma\) is torsion-free, \(r'(\delta \sigma) = 1\) and \(\delta \sigma \in \Ker r' \subseteq \Delta\). Thus, \((\delta \sigma)^m \ne 1\), as \(\Delta \subseteq \Gamma'' \subseteq \Gamma'\) is torsion-free.
\end{proof}

Finally, to prove Lemma~\ref{lemma: inductive step}, we use the notions introduced at the beginning of this section. Let us apply the theorem above to \(\Gamma = \Gamma^n\), \(\Sigma = \Gamma^{n - 1}_1\), and \(\Gamma^n_1 \subseteq \Gamma^n\). The resulting subgroup \(\Delta' = \Gamma^n_1\) satisfies the required conditions.

\bibliographystyle{alpha}
\bibliography{refs.bib}

@article{BGS20,
	title           = {Homology and homotopy complexity in negative curvature},
	volume          = {22},
	ISSN            = {1435-9855},
	DOI             = {10.4171/jems/971},
	abstractNote    = {Uri Bader, Tsachik Gelander, Roman Sauer},
	number          = {8},
	journal         = {Journal of the European Mathematical Society},
	author          = {Bader, Uri and Gelander, Tsachik and Sauer, Roman},
	year            = {2020},
	month           = {may},
	pages           = {2537–2571},
	language        = {en}
}

@article{BHW11,
	title           = {Hyperplane sections in arithmetic hyperbolic manifolds},
	volume          = {83},
	rights          = {© 2011 London Mathematical Society},
	ISSN            = {1469-7750},
	DOI             = {10.1112/jlms/jdq082},
	abstractNote    = {We prove that the fundamental groups of ‘standard’ arithmetic hyperbolic manifolds virtually retract onto their geometrically finite subgroups. In particular, this implies that the homology groups of immersed totally geodesic hypersurfaces of compact arithmetic hyperbolic manifolds virtually inject in the homology groups of the ambient manifold.},
	number          = {2},
	journal         = {Journal of the London Mathematical Society},
	author          = {Bergeron, Nicolas and Haglund, Frédéric and Wise, Daniel T.},
	year            = {2011},
	pages           = {431–448},
	language        = {en}
}

@article{Bug84,
	author 			= {Bugaenko, V. O.},
	title 			= {Groups of automorphisms of unimodular hyperbolic quadratic forms over the ring $\mathbf{Z}\bigl[\frac{\sqrt5+1}2\bigr]$},
	journal 		= {Vestnik Moskov. Univ. Ser.~1. Mat. Mekh.},
	year 			= {1984},
	issue			= {5},
	pages			= {6--12},
	zmath			= {https://zbmath.org/?q=an:0571.10024}
}

@article{BV13,
	title           = {The asymptotic growth of torsion homology for arithmetic groups},
	volume          = {12},
	ISSN            = {1474-7480, 1475-3030},
	DOI             = {10.1017/S1474748012000667},
	abstractNote    = {When does the amount of torsion in the homology of an arithmetic group grow exponentially with the covolume? We give many examples where this is the case, and conjecture precise conditions.},
	number          = {2},
	journal         = {Journal of the Institute of Mathematics of Jussieu},
	author          = {Bergeron, Nicolas and Venkatesh, Akshay},
	year            = {2013},
	month           = {apr},
	pages           = {391–447},
	language        = {en}
}

@article{PV05,
	title           = {On right-angled reflection groups in hyperbolic spaces},
	volume          = {80},
	ISSN            = {0010-2571},
	DOI             = {10.4171/cmh/4},
	abstractNote    = {Leonid Potyagailo, Ernest Vinberg},
	number          = {1},
	journal         = {Commentarii Mathematici Helvetici},
	author          = {Potyagailo, Leonid and Vinberg, Ernest},
	year            = {2005},
	month           = {mar},
	pages           = {63–73},
	language        = {en}
}

@article{Rai12,
	title           = {Exponential growth of torsion in abelian coverings},
	volume          = {12},
	DOI             = {10.2140/agt.2012.12.1331},
	abstractNote    = {We show exponential growth of torsion numbers for links whose first nonzero Alexander polynomial has positive logarithmic Mahler measure. This extends a theorem of Silver and Williams to the case of a null first Alexander polynomial and provides a partial solution for a conjecture of theirs. 57M10; 57M25, 57Q10},
	number          = {3},
	journal         = {Algebraic and Geometric Topology},
	publisher       = {Mathematical Sciences Publishers},
	author          = {Raimbault, Jean},
	year            = {2012},
	pages           = {1331–1372}
}

@book{Rat19,
	address         = {Cham},
	series          = {Graduate Texts in Mathematics},
	title           = {Foundations of Hyperbolic Manifolds},
	volume          = {149},
	rights          = {http://www.springer.com/tdm},
	ISBN            = {978-3-030-31596-2},
	url             = {http://link.springer.com/10.1007/978-3-030-31597-9},
	DOI             = {10.1007/978-3-030-31597-9},
	publisher       = {Springer International Publishing},
	author          = {Ratcliffe, John G.},
	year            = {2019},
	collection      = {Graduate Texts in Mathematics},
	language        = {en}
}

@article{Ves17,
	title           = {Right-angled polyhedra and hyperbolic 3-manifolds},
	volume          = {72},
	rights          = {http://iopscience.iop.org/info/page/text-and-data-mining},
	ISSN            = {0036-0279, 1468-4829},
	DOI             = {10.1070/RM9762},
	number          = {2},
	journal         = {Russian Mathematical Surveys},
	author          = {Vesnin, A. Yu.},
	year            = {2017},
	month           = {apr},
	pages           = {335–374}
}

\end{document}